\documentclass[final, 12pt]{amsart}
\usepackage{amsmath,amssymb,amsfonts,amsthm,slashed,braket}
\usepackage[all]{xy}
\usepackage{mathrsfs}
\usepackage{fancyhdr}
\usepackage{showkeys}
\usepackage{hyperref}

\title{A combinatorial divisibility question from noncommutative algebra}
\author{Arnav Tripathy}

\newcommand\nc{\newcommand}
\nc\linesep{\bigskip}
\nc\newprob[1]{\marginnote{#1}[\parskip]}
\nc\bA{\mathbb A}
\nc\bC{\mathbb C}
\nc\bD{\mathbb D}
\nc\bR{\mathbb R}
\nc\bZ{\mathbb Z}
\nc\bQ{\mathbb Q}
\nc\bP{\mathbb P}
\nc\bV{\mathbb V}
\nc\bW{\mathbb W}
\nc\bG{\mathbb G}
\nc\mf\mathfrak
\nc\mc\mathcal
\nc\mb\mathbb
\nc\brac[1]{\langle#1\rangle}
\nc\abs[1]{\lvert#1\rvert}
\nc\norm[1]{\lVert#1\rVert}
\nc\onto{\twoheadrightarrow}
\nc\into{\hookrightarrow}
\nc\lto{\longrightarrow}
\nc\action{\curvearrowright}
\nc\on\operatorname
\nc\wbar\overline
\nc\what\widehat
\nc\wtilde\widetilde
\nc\nop\DeclareMathOperator
\nc\eps{\varepsilon}
\nc\tsym{\widetilde{\text{Sym}}}
\nc\oarrow[1]{\overset{#1}\to}
\nop\Hom{Hom}
\nop\End{End}
\nop\Aut{Aut}
\nop\im{Im}
\nop\id{id}
\nop\tr{Tr}
\nop\coker{coker}
\nop\Spec{Spec}
\nop\Jac{Jac}
\nop\Ext{Ext}
\nop\Tor{Tor}
\nc\op{\text{op}}
\nop\loc{Loc}
\nop\Frac{Frac}
\nc\ann{\text{ann}}
\nop\QCoh{QCoh}
\nop\Coh{Coh}
\nop\Sym{Sym}
\nop\gr{Gr}
\nop\Tot{Tot}
\nop\Fl{Fl}
\nop\tGamma{\widetilde\Gamma}
\nop\tloc{\widetilde{\text{Loc}}}
\nop\rep{Rep}
\nop\proj{Proj}
\nc\oo[1]{\overset\circ{#1}}
\nop\ospec{\oo{Spec}}
\nop\oTot{\oo{Tot}}
\nop\Bl{Bl}
\nop\Comp{Comp}
\nop\Ho{Ho}
\nop\cone{Cone}
\nop\LKE{LKE}
\nop\RKE{RKE}
\nop\pd{pd}
\nop\cd{cd}
\nop\depth{depth}
\nop\ass{Ass}
\nop\supp{supp}
\nop\codim{codim}
\nop\holim{\underset{\lto}{holim}}
\nop\dlim{\underset{\lto}{lim}}

\nop\uHom{\underline{\Hom}}
\nop\Pic{Pic}
\nop\Cl{Cl}
\nop\Div{Div}
\nop\rank{rank}
\nop\Der{Der}
\nop\dimrel{dim.rel}
\nc\sHom{\mathscr Hom}
\nc\sExt{\mathscr Ext}
\nc\dto{\dashrightarrow}
\nop\rspec{\bf Spec}
\nop\Gal{Gal}
\nop\Ind{Ind}
\nop\Frob{Frob}
\nop\Fib{Fib}
\nop\ratdim{rat\ dim}
\nop\Mod{Mod}
\nop\rat{rat}
\nop\val{val}
\nop\Rep{Rep}
\nop\colim{colim}
\nop\ind{ind}

\theoremstyle{theorem}
\newtheorem{thm}{Theorem}

\newtheorem{defn}[thm]{Definition}
\newtheorem{prop}[thm]{Proposition}
\newtheorem{conj}[thm]{Conjecture}

\theoremstyle{remark}
\newtheorem{rem}[thm]{Remark}

\begin{document} 
\maketitle
\tableofcontents

\begin{abstract}
We present a general conjecture on the divisibility of a certain expression in terms of Kostka numbers and their close variants. This conjecture is closely related to a variant of the period-index problem of noncommutative algebra, with partial implications in both directions. We present a description of the connection between these two problems via Schubert calculus as motivation and evidence for the conjecture before turning to a proof of the conjecture in a family of cases.
\end{abstract}

\section{Introduction}

Kostka numbers and their relatives are among the most important coefficients that arise in representation theory, yet despite being of obvious combinatorial importance, their divisibility properties are poorly known. Contrast this situation to, for example, that of binomial coefficients, where Lucas's and Kummer's theorem among others give us extremely precise information on their divisibility. We present in this paper a series of combinatorial divisibility conjectures on these numbers that have their birth in an intersection-theoretic approach to variants of the period-index problem in noncommutative algebra. Besides proving the conjectures in a special case, we also include an explanation of the motivation from noncommutative algebra that moreover establishes a lower bound in a larger family of cases, as reasonable evidence for the conjectures beyond the cases proven here and the numerical verification of small cases via Sage.

We now state our main conjectures. Given a partition $\lambda$, denote by $s_{\lambda}$ the Schur polynomial of type $\lambda$ and consider an inner product $\langle - , - \rangle$ given by taking the Schur polynomials as an orthonormal basis; this construction is the usual inner product on the space of class functions of the general linear group. We consider the following integer: \begin{defn} Given two positive integers $m$ and $n = km$, denote by $$g(m, n) = \langle s_{m^{n-m}} s_{(n-m)^m}, s_{m^m}^{2(k-1)} \rangle.$$ \end{defn} Note in the above that by $m^n$, for example, we mean the Young tableau consisting of a rectangle with $n$ rows and $m$ columns. Here, then, is our main conjecture: \begin{conj} Given a prime $p$ and positive integers $e < f$, and setting $m = p^e, n = p^f$, we have $$\val_p g(m, n) = f - e.$$ \end{conj} In fact, using the Cauchy identity as in~\cite{Bump}, we may rewrite $g(m, n)$ in terms of a single coefficient of a symmetric polynomial expression, now in two infinite sets of variables. Indeed, if $\Lambda$ denotes the ring of symmetric functions, we consider the basis of $\Lambda \hat{\otimes} \Lambda$ given by $s_{\lambda} s'_{\mu}$, where the primed Schur polynomials refer to the second set of variables. We then have the following application of the Cauchy identity: \begin{prop} We may rewrite $$g(m, n) = \langle \Big( \sum_{\lambda \subset (n-m)^m} s_{\lambda} s'_{\tilde{\lambda}^*} \Big)^{2(k-1)}, s_{(n-m)^m} s'_{(n-m)^m} \rangle.$$ \end{prop} As we discuss further below, here by $\lambda \subset (n-m)^m$ we mean to sum over partitions $\lambda$ which fit in an $m$ by $n-m$ rectangle, and then $\tilde{\lambda}^*$ denotes the transpose dual. In any case, this reformulation yields a corresponding reformulation of our main conjecture, and we conjecture a combinatorial divisibility relation that would immediately imply, in much stronger form, the correct lower bound for our main conjecture. Our conjecture here is the following: \begin{conj} For $m = p^e, n = p^f$ as before, if one expands $\sum_{\lambda \subset (n-m)^m} s_{\lambda} s'_{\tilde{\lambda}^*}$ in the elementary basis and then takes the multinomial expansion of the $2(p^{f-e}-1)$ power, then in fact every term, upon pairing with $s_{(n-m)^m} s'_{(n-m)^m}$, is divisible by $p^{f-e}$. \end{conj} 


We rewrite the above to make more manifest the appearance of Kostka numbers and their variants. Recall, among the many natural bases for the symmetric function algebra, the Schur polynomials $\{s_{\lambda}\}$, the elementary polynomials $\{e_{\lambda}\}$, and the monomial basis $\{m_{\lambda}\}$. We denote the various change-of-basis matrices with the appropriate superscripts, such as $M^{se}$ the change-of-basis matrix from the Schur to the elementary basis. We denote particular matrix elements by subscripts, so that for example we write $$s_{\lambda} = \sum_{\mu} M^{se}_{\lambda \mu} e_{\mu}.$$ With this notation, the Kostka matrix $K$ with matrix elements the Kostka numbers $K_{\lambda \mu}$ is the change-of-basis matrix $M^{sm} = K$ while the other matrix coefficients we used are naturally expressible in terms of Kostka numbers via the relations $$M^{se} = J (K^T)^{-1}, M^{es} = K^T J,$$ where $J_{\lambda \mu} = \delta_{\tilde{\lambda} \mu}$ is a transposition matrix; see for example I.6 of~\cite{Macdonald} for details. We may hence reformulate the above conjecture as the following: \begin{conj} For every tuple $\{c_{\lambda \mu \nu}\}$ summing to $2(k-1)$, we have $$p^{f-e} \Big| \binom{2(k-1)}{\{c_{\lambda \mu \nu}\}} \Big( \prod_{\lambda, \mu, \nu} M^{se}_{\lambda \mu} M^{se}_{\tilde{\lambda}^* \nu} \Big) M^{es}_{\sum_{\lambda, \mu, \nu} c_{\lambda \mu \nu} \mu, (n-m)^m} M^{es}_{\sum_{\lambda, \mu, \nu} c_{\lambda \mu \nu} \nu, (n-m)^m}.$$ \end{conj} 

Of course, our first conjecture could also be reformulated directly in terms of these combinatorial quantities. It is these formulations that focus attention on the divisibility properties of Kostka numbers and their variants, and we would be highly interested in methods suitable for resolving these conjectures, either within or outside of pure combinatorics. In particular, the natural interpretation of Kostka numbers in terms of the representation theory of the symmetric or general linear groups suggests that one fruitful approach may be via characteristic $p$ modular representation theory. A robust theory that allowed one to easily prove divisibility of such coefficients would open the pathway to a generally applicable tool in the theory of central simple algebras, and conversely ``generic'' constructions in central simple algebras could yield a host of divisibility bounds. 

This paper arose due to the surprisingly rich combinatorial structure of an algebro-geometric approach to the period-index problem in noncommutative algebra. We address the original noncommutative algebra and algebraic geometry motivation and results in section $2$ before explaining in section $3$ their relation to the combinatorial conjectures above. Finally, in section $4$, we address the special case of $p = 2, e = 1$ in the above conjectures, where we give a lengthy but direct proof. We first learned of this problem and its associated geometry from Daniel Krashen at the 2014 Workshop in Algebraic Geometry at Seattle, to whom we owe liberal thanks; we also thank the organizers for the stimulating environment they provided at the workshop. We thank Daniel Bump, Angela Hicks, and Persi Diaconis for helpful conversations and context. Further thanks are due to Zeb Brady for frequent insights and motivation, and finally to Ravi Vakil for copious advice and comments. 

\section{Index reduction of central simple algebras}

Recall that a central simple algebra $A$ over a field $k$ is a finite-dimensional associative algebra that is simple as a $k$-algebra, with center the ground field $k$. In fact, we have $\dim_k A = n^2$ for some positive integer $n$, which we call the degree of $A$, and $A$ is said to be split if it is isomorphic to $Mat_n(k)$, the $n \times n$ matrices over $k$. A field extension $E / k$ is a splitting field for $A$ if $A \otimes_k E \simeq Mat_n(E)$, and a central problem in noncommutative algebra concerns how large splitting fields need to be given basic invariants of a central simple algebra $A$. This number is known as the index $\ind(A)$, and may also be defined by the following fact: a central simple algebra $A$ is always isomorphic to a matrix algebra over a division algebra; the degree of this associated division algebra is also $\ind(A)$.

The problem we study here is more nuanced than estimating the index of a single algebra. Instead, we take a pair of central simple algebras $A_1, A_2$ and consider a variant of the question as to whether they have a common subfield of given degree over $k$. To phrase the question more precisely, we introduce the generalized Severi-Brauer variety associated to a central simple algebra:

\begin{defn} Given a central simple algebra $A$, define the $m$th generalized Severi-Brauer variety $X_m(A)$ by the functor of points $$X_m(A)(R) = \{ \text{rank }(m \cdot \deg A)\text{ right ideals which are direct summands of }A \otimes_k R \}.$$ \end{defn}

The above functor is representable by a scheme which is a form of the Grassmannian $G(m, \deg A)$ over the non-algebraically closed field $k$. Alternatively, a central simple algebra of degree $n$ determines a Galois cohomology class $H^1(k, PGL_n)$, but as we have an obvious morphism $PGL_n \to \Aut(G(m, n))$, this class automatically defines a twisted form of the Grassmannian over $k$. For further details on these constructions and the argument that they agree, see~\cite{KrashenZero}.

These generalized Severi-Brauer varieties control the index reduction problem, a generalization of problems on the index itself. Recall the index is the degree of a field extension $E/k$ that would completely split the central simple algebra $A$, i.e. passing to $A \otimes_k E$ would reduce the index to $1$. Instead, we can ask for field extensions that merely reduce the index to some specified smaller number. From the definition, we immediately have that $X_m(A)$ has a rational point over $E$ if and only if $\ind(A \otimes_k E)$ divides $m$. Hence, one approach to the index reduction problem is to study the index of this variety, where the index of a variety is defined as the greatest common divisor of the degree of all rational zero-cycles. Analyzing $\ind X_m(A)$ is usually more algebro-geometrically tractable and provides a lower bound for the degree of any field extension that would reduce the index to $m$. As we are concerned here with a pair of algebras, our motivating question will be to estimate $\ind X_m(A_1) \times X_m(A_2)$, which now gives a bound for field extensions which simultaneously reduces the indices of both $A_1$ and $A_2$ to (a divisor of) $m$. 

We focus on the particular case where $A_1$ and $A_2$ both have degree and index equal to $n = 2^f$. We moreover assume that $A_1 \otimes A_2$ has index $2$. We seek to simultaneously reduce the indices of $A_1$ and $A_2$ to some smaller power $m = 2^e$. Then, the following result is due to~\cite{KrashenCorestriction}: \begin{thm} For central simple algebras $A_1, A_2$ as above, we always have $$\ind \Big(X_m(A_1) \times X_m(A_2)\Big) \bigm| 2^{f-e}.$$ Moreover, this bound is sharp for a suitably generic choice of $A_1$ and $A_2$. \end{thm}

Our combinatorial reinterpretation of this question, which we begin in the next section, is directly motivated by this question with partial implications in both directions. A successful resolution of our combinatorial conjecture for the case $p = 2$ would constitute a new proof of the first statement in the above theorem, namely the universal bound $2^{f-e}$ for the index. In the opposite direction, the fact that Krashen showed the existence of suitable central simple algebras for which the index was no smaller does indeed establish one half of our combinatorial conjecture in the case $p = 2$, namely that the $2$-valuation is at least $f - e$. As such, we interpret the tie of our combinatorial conjecture to the noncommutative algebra literature as further favorable evidence. 

\section{From algebra to combinatorics}

We bound the index of $X_m(A_1) \times X_m(A_2)$ via intersection theory. Indeed, we construct a zero-cycle rational over the base field as the intersection of higher-dimensional cycles more obviously rational over the base field before computing the resulting degree, as a bound for the index, via intersection theory. In other words, if we denote by $V$ our variety $X_m(A_1) \times X_m(A_2)$, we work in the Chow ring and consider cycles in the image of $$CH^*(V) \to CH^*(V_{\overline{k}}) \simeq CH^*(G(m, n) \times G(m, n)),$$ before intersecting those cycles sufficiently many times to get a top-degree cycle, which we then evaluate under the degree map $$\deg: CH_0(V) \to CH_0(V_{\overline{k}}) \to \mb{Z}.$$ It is this degree $g(m, n)$ that will provide us our bound on the index and for which we wish to estimate the $p$-valuation. 

Note in the above that the natural map $CH^*(V) \to CH^*(V_{\overline{k}})$ is given by pullback under the base-change to the algebraic closure $V_{\overline{k}} \to V$, and that as $X_m(A_i)$ and $G(m, n)$ become isomorphic over the algebraic closure, we identify $V_{\overline{k}}$ with $G(m, n) \times G(m, n)$. Next, in order to construct Chow cycles that are rational over the original ground field, we recall the result of Artin in~\cite{ArtinBrauer} that given a central simple algebra $A$ of index $d$, the class of a codimension $d$ hyperplane section under the Plucker embedding of any generalized Severi-Brauer variety is well-defined over the base field. In the example $p = 2$ relevant for the connection to noncommutative algebra, this implies that if we consider the Segre map $$X_m(A_1) \times X_m(A_2) \to X_{m^2}(A_1 \otimes_k A_2),$$ the pullback of the appropriate power of the hyperplane class from the target is indeed rational over the base field and so suffices for the application at hand. Hence, we turn to the study of this intersection-theoretic problem. 

At this point, we may as well work purely over the algebraic closure, as the degree is certainly ambivalent as to over which field we work. Hence, we wish to calculate the pullback in Chow of the Schubert class $\phi^* \sigma_{1^{mn}}$ under the map \begin{eqnarray*} G(m, V) \times G(n, W) &\stackrel{\phi}{\to}& G(mn, V \otimes W) \\ (A \subset V, B \subset W) &\mapsto& \Big( A \otimes B \subset V \otimes W \Big). \end{eqnarray*} We will remain in characteristic zero for now so as to rely on the crutch of Kleiman's transversality theorem, but as the results are universal, they continue to hold in any characteristic. In the sequel, we will work with the intersection theory of the Grassmannian throughout, and hence use $\sigma_{\lambda}$ to denote the relevant basis of Schubert cycles, but recall that under the usual identification of the Chow ring of the Grassmannian with the (truncated) ring of symmetric functions, these Schubert cycles $\sigma_{\lambda}$ correspond precisely to the Schur polynomials $s_{\lambda}$, as explained for example in~\cite{Tamvakis}. Hence, the following theorem shows that we may indeed identify the numbers $g(m, n)$ defined in the introduction with the degree of the top power of the following pullback.

\begin{thm} (i) If $V, W$ both have dimension $m + n$ such that the map is $G(m, m+n) \times G(n, m+n) \to G(mn, (m+n)^2)$, then we have the following formula for the pullback: $$ \phi^* \sigma_{1^{mn}} = \sum_{\lambda} \sigma_{\lambda} \sigma_{\tilde{\lambda}^*}'. $$ (ii) If $V$ and $W$ have dimension at least $m + n$ or even in the limit of considering the map $G(m, \infty) \times G(n, \infty) \to G(mn, \infty),$ the same formula still holds provided one continues to only sum over such Young tableaux $\lambda$ as fit inside an $m \times n$ rectangle and the Poincar\'{e} dual cycle is still taken with respect to an $m \times n$ rectangle. \end{thm}

\begin{rem} Some remarks about notation are in order. Here, when discussing the Chow ring of $G(m, V) \times G(n, W)$, we denote Schubert cycles from the first factor as unprimed and the second factor as primed, so a typical basis element of Chow would be denoted $\sigma_{\lambda} \sigma_{\mu}'$. Next, given a partition or Young tableaux $\lambda$, we denote by $\tilde{\lambda}$ its transpose, following~\cite{FultonYoung}; for example, under the duality $$\iota: G(m, m+n) \, \tilde{\to} \, G(n, m+n),$$ we have $\iota^* \sigma_{\lambda} = \sigma_{\tilde{\lambda}}$. Finally, we denote the dual tableau to $\lambda$ by $\lambda^*$; concretely, if we are working in $G(m, m+n)$, we have $$\lambda^*_i = n + 1 - \lambda_{m + 1 - i}$$ so that $\sigma_{\lambda}$ and $\sigma_{\lambda^*}$ represent Poincar\'{e} dual cycles. \end{rem}

\begin{proof} We start with the second part. The content of this statement is that when we form $\phi^* \sigma_{1^{mn}}$ for the map $\phi: G(m,\infty) \times G(n,\infty) \to G(mn,\infty)$, no terms in the result vanish upon restricting to the Chow ring of $G(m,m+n) \times G(n,m+n)$. In other words, when we expand $\phi^* \sigma_{1^{mn}}$ into tensor products of Schubert cycles, all the terms are of the form $\sigma_{\lambda} \sigma'_{\mu}$ where $\lambda$ fits in an $m \times n$ rectangle while $\mu$ fits in an $n \times m$ rectangle (i.e. following the convention of~\cite{FultonYoung}, $\lambda$ should have at most $m$ rows and at most $n$ columns while the reverse holds for $\mu$). 

Now, certainly all Young tableau $\lambda$ which appear will have no more than $m$ rows as those are the classes that span the Chow ring of $G(m, \infty)$; the nontrivial assertion is that all classes which appear have at most $n$ columns, and analogously for $\mu$. Consider the tautological sub-bundle $\mc{S}_{mn}$ on $G(mn, \infty)$; by the construction of the map $\phi$, it follows that $$\phi^* \mc{S}_{mn} \simeq \mc{S}_m \otimes \mc{S}_n,$$ where the notation for the tautological sub-bundles of our two factors follows our usual convention. In fact, it will be more helpful to dualize this equation and instead write $$\phi^* \mc{S}_{mn}^* \simeq \mc{S}^*_m \otimes \mc{S}^*_n.$$ Recall that the Chern classes of the duals of these tautological sub-bundles are given by, for example, $c_i \mc{S}^*_m = \sigma_{1^i}$ for $i \le m$. We are hence interested in evaluating $$\phi^* \sigma_{1^{mn}} = \phi^* c_{mn} \mc{S}_{mn} = c_{mn} (\mc{S}^*_m \otimes \mc{S}^*_n),$$ at which point we recall that the top Chern class of a tensor product is given by a resultant as follows, possibly up to a unit: $$\phi^* \sigma_{1^{mn}} = \pm \begin{vmatrix} 1 & -\sigma_1 & \sigma_{11} & \cdots & (-1)^m \sigma_{1^m} & 0 & 0 & \cdots & 0 \\ 0 & 1 & -\sigma_1 & \cdots & (-1)^{m-1} \sigma_{1^{m-1}} & (-1)^m \sigma_{1^m} & 0 & \cdots & 0 \\ \vdots & \vdots & \vdots && \vdots & \vdots & \vdots && \vdots \\ 0 & 0 & 0 & \cdots & 1 & -\sigma_1 & \sigma_{11} & \cdots & (-1)^m \sigma_{1^m} \\ 1 & \sigma_1' & \sigma_{11}' & \cdots &  \sigma'_{1^{n-1}} & \sigma_{1^n}'  & 0 & \cdots & 0 \\ 0 & 1 & \sigma_1' & \cdots & \sigma_{1^{n-2}}' & \sigma_{1^{n-1}}' & \sigma_{1^n}' & \cdots & 0 \\ \vdots & \vdots & \vdots && \vdots & \vdots & \vdots && \vdots \\ 0 & 0 & \cdots & 1 & \sigma_1' & \sigma_{11}' & \cdots & \cdots & \sigma_{1^n}' \end{vmatrix}. $$ We immediately see from this expression that, upon expanding, each monomial has at most $n$ factors of the form $\sigma_{1^i}$ and $m$ factors of the form $\sigma'_{1^j}$ so that when we use Pieri's formula to multiply in the Schubert ring, each monomial only has partitions with at most $n$ columns for the first factor and $m$ columns for the second factor, as desired.

We now turn to the first statement of the theorem. \begin{equation*} \xymatrix{ & \phi^{-1}(Z) \ar@{^{(}->}[rr] \ar@{^{(}->}[d] & & Z \ar@{^{(}->}[d]  \\ & G(m, V) \times G(n, W) \ar@{^{(}->}[rr]^{\phi} \ar[dl]_{\pi_1} \ar[dr]^{\pi_2} & & G(mn, V \otimes W) \\ G(m, V) & & G(n, W)  } \end{equation*} We again denote by $V$ and $W$ our ambient vector spaces for our two Grassmannian factors where now $\dim V = \dim W = m + n$. We compute $\phi^* \sigma_{1^{mn}}$ by picking a representative cycle $Z$ inside $G(mn, V \otimes W)$, intersecting that cycle with the Segre-embedded $G(m, V) \times G(n, W)$ inside $G(mn, V \otimes W)$, and then viewing the resulting cycle $\phi^{-1}(Z)$ as a correspondence between $G(m, V)$ and $G(n, W)$ that induces a map on Chow groups in the usual way. Namely, given some class $\sigma \in CH^*(G(n, W))$, its image under the map induced by $\phi^{-1}(Z)$ is $$\sigma \mapsto (\pi_1)_* \Big( \phi^{-1}(Z) \cdot \pi_2^* \sigma \Big).$$ Computing this map will straightforwardly lead to an expression for $\phi^* \sigma_{1^{mn}}$, i.e. the class of $\phi^{-1}(Z)$ in Chow.

We must be careful that our choice of representative cycle inside $G(mn, V \otimes W)$ be transverse to the embedded $G(m, V) \times G(n, W)$ and that the subsequent intersection be transverse to the preimages of the Schubert cycle representatives we will choose inside the second factor. Well, a typical representative of $\sigma_{1^{mn}}$ inside $G(mn, V \otimes W)$ would be to choose an element $\xi \in (V \otimes W)^*$ and consider the locus of $mn$-dimensional subspaces annihilated by $\xi$. Using Kleiman's transverality theorem, this representative $Z_{\xi}$ will indeed be transverse to any other smooth subscheme provided we take $\xi$ generic, which we shall now do and thereby satisfy all our transversality requirements. In particular, we suppose that $\xi$ has full rank, i.e. if we consider it as a morphism $\xi: W \to V^*$, it is an isomorphism. 

Under the projections $\pi_1, \pi_2$ from $G(m, V) \times G(n, W)$ to its component factors, we claim that the intersection $\phi^{-1}(Z_{\xi})$ projects isomorphically to either factor and is thereby the graph of an isomorphism between the two varieties. We will argue this claim set-theoretically, leaving to the reader the notational upgrade required to perform the argument on the functor-of-points level. Indeed, given $A \subset V$ of dimension $m$, any preimage under $\pi_1: \phi^{-1}(Z_{\xi}) \to G(m, V)$ is a point $B \subset W$ satisfying $\xi(A \otimes B) = 0 \Leftrightarrow B \subset (\xi(A))^{\perp}$, which is already of dimension $n$ so that there exists a unique such subspace $B$. We may similarly show $\pi_2: \phi^{-1}(Z_{\xi}) \to G(n, W)$ is an isomorphism. Hence, $\phi^{-1}(Z_{\xi})$ is the graph of the isomorphism \begin{eqnarray*} G(m, V) &\stackrel{\iota}{\to}& G(n, W) \\ A &\mapsto& (\xi(A))^{\perp}. \end{eqnarray*} It hence follows that the ensuing map on Chow groups induced by this correspondence is simply given by $\sigma'_{\mu} \mapsto \sigma_{\tilde{\mu}}$. It is now immediate to verify that the formula we provided performs exactly this morphism when interpreted as a correspondence between the two varieties in question. \end{proof}

Return to the special case of $G(m, m + n ) \times G(n, m + n)$ and consider the coefficient of the point class in $(\phi^* \sigma_{1^{mn}})^2$. As each term in the expansion above now has a unique partner term (possibly itself) to which it is Poincar\'{e} dual, this coefficient is simply the total number of terms, or the total number of Young tableaux that fit inside an $m \times n$ rectangle. As we may put this set in bijection with the set of paths from the lower-left corner of the rectangle to the upper-right corner using only up and right moves (the Young tableaux then being the portion of the rectangle to the upper-left), this coefficient is immediately seen to be $\binom{m+n}{m, n}$. As such, we have the following:

\begin{thm} Conjecture 2 is true for the special case $p = 2$ and $f = e + 1$. \end{thm}

\begin{proof} The above argument explicitly calculates $$g(2^e, 2^{e+1}) = \binom{2^{e+1}}{2^e},$$ which is only singly divisible by $2$ by Kummer's theorem, and so $$\val_p g(m, n) = 1 = f - e.$$ \end{proof}

In the next section, we treat the opposite special case, where $e$ is fixed at $1$ but $f$ may be arbitrarily far apart from $e$.

\section{The case $p^e = 2$}

In this section, we shall show the following.

\begin{thm} Conjectures 2 and 4 are true for the special case $p = 2$ and $e = 1$. \end{thm}

In fact, we establish the more general claim that $$\val_2 g(2, 2\ell+2) = v_2(\ell),$$ where $v_2(\ell)$ denotes the number of $1$s in the binary expansion of $\ell$. Similarly, we will more generally show that if we expand via the multinomial formula in the elementary basis, every term that contributes to $g(2, 2\ell + 2)$ has $2$-valuation at least $v_2(\ell)$, so that we will establish analogues of both Conjectures 2 and 4 in this slightly greater generality.

To make progress on the evaluation of $g(2,2\ell+2)$, we begin by rewriting \begin{eqnarray*} \phi^* \sigma_{1111}^\ell &=& (\sigma_1^2 \sigma_{11}' + \sigma_{11} \sigma_1'^2 - 2 \sigma_{11} \sigma_{11}' + \sigma_1 \sigma_{11} \sigma_1' + \sigma_1 \sigma_1' \sigma_{11}' + \sigma_{11}^2 + \sigma_{11}'^2)^{2\ell} \\ &=& \sum_{\substack{a + b + c + d + \\e + f + g = 2\ell}} (-2)^c \binom{2\ell}{a,b,c,d,e,f,g} \sigma_1^{2a + d + e} \sigma_1'^{2b + d + e} \sigma_{11}^{b + c + d + 2f} \sigma_{11}'^{a + c + e + 2g}. \end{eqnarray*} Note that the above is precisely the expansion of the original expression in the Schur classes in terms of the elementary basis. Recall that we are seeking the coefficient of the point class $\sigma_{2\ell,2\ell} \sigma_{2\ell,2\ell}'$, which we will denote using the notation for extracting coefficients of certain terms in generating functions, i.e. by writing for example $[\sigma_{2\ell,2\ell} \sigma_{2\ell,2\ell}']$ before an expression. Recall furthermore that multiplying by $\sigma_{11}$ in $CH^*(G(2,\infty))$ simply raises both indices by $1$, so that we may now write \begin{eqnarray*} [\sigma_{11}^{2\ell} \sigma_{11}'^{2\ell}] \phi^* \sigma_{1111} &=& \sum_{\substack{a+b+c+d+\\e+f+g=2\ell}} (-2)^c \binom{2\ell}{a,b,c,d,e,f,g} [\sigma_{11}^{2\ell-b-c-d-2f}\sigma_{11}'^{2\ell-a-c-e-2g}] \sigma_1^{2a+d+e} \sigma_1'^{2b+d+e} \\ &=& \sum_{\substack{a+b+c+d+\\e+f+g=2\ell}} \Big( (-2)^c \binom{2\ell}{a,b,c,d,e,f,g} \delta_{2a+d+e,2(2\ell-b-c-d-2f)} \delta_{2b + d + e,2(2\ell-a-c-e-2g)} \Big) \\ && \qquad \qquad \qquad \qquad \times \Big( [\sigma_{11}^{a + (d + e)/2}] \sigma_1^{2a + d + e} \Big) \times \Big( [\sigma_{11}'^{b + (d + e)/2}] \sigma_1'^{2b + d + e} \Big) \\ &=& \sum_{\substack{a+b+c+d+\\e+f+g=2\ell}} (-2)^c \binom{2\ell}{a,b,c,d,e,f,g} \delta_{d + 2f, e + 2g} C_{a + (d + e)/2} C_{b + (d + e)/2}, \end{eqnarray*} where we made several simplifications throughout the above. First, we noted that the only way the power $\sigma_1^{2a + d + e}$ could provide a $\sigma_{11}^{2\ell-b-c-d-2f}$ term in the first Chow ring factor would be if the degrees matched, i.e. if $$2a + d + e = 2(2\ell - b - c - d - 2f),$$ but $$2\ell - b - c - d - 2f = a + e + g - f$$ so that condition turns into \begin{eqnarray*} 2a + d + e &=& 2(a + e + g - f) \\ \Leftrightarrow d + 2f &=& e + 2g. \end{eqnarray*} Similarly, the other Kronecker delta collapses to the same condition (as it must, as one may equivalently phrase the condition as simply that the two Chow ring factors contribute equal degree to the term in question). Note that this condition also implies $d$ and $e$ have the same parity so $(d + e)/2$ is an integer and the expressions on the subsequent lines make sense. Finally, it remains to recall that in $CH^*(G(2,k+2))$, the coefficient $[\sigma_{11}^k] \sigma_1^{2k}$ (of classical interest as the degree of the Pl\"{u}cker embedding) is given by the Catalan number $$C_k = \frac{1}{k+1} \binom{2k}{k}$$. We hence now have the task of evaluating the $2$-valuation of the sum $$\sum_{\substack{a + b + c + d + e + f + g = 2m \\ d + 2f = e + 2g}} (-2)^c \binom{2m}{a,b,c,d,e,f,g} C_{a + (d + e)/2} C_{b + (d + e)/2}.$$ 

As claimed above, we will in fact show that every single term in this sum is actually divisible by $2^{v(\ell)}$ in accord with the second of our main conjectures. We hence have two tasks ahead of us now: first, to establish this claim that every term in the sum above has $2$-valuation at least $v(\ell)$ and second, to show that we have an odd number of terms with $2$-valuation exactly $v(\ell)$. To evaluate the $2$-valuations of the terms appearing, recall that from Kummer's theorem, we have $$\val_2 \binom{n}{k} = v(k) + v(n-k) - v(n)$$ and that iterating this construction gives, for example, that $$\val_2 \binom{2\ell}{a,b,c,d,e,f,g} = v(a) + v(b) + v(c) + v(d) + v(e) + v(f) + v(g) - v(2\ell).$$ We also have \begin{eqnarray*}\val_2 C_k &=& \val_2 \binom{2k}{k} - \val_2 (k+1) \\ &=& v(k) + v(k) - v(2k) - \val_2 (k+1) \\ &=& v(k) - \val_2(k+1) = v(k+1) - 1,\end{eqnarray*} where the last equality follows by pondering the binary expansion for $k$: a priori, $\val_2(k+1)$ is the number of $1$s in the terminal string of consecutive $1$s in the binary expansion of $k$ and so $v(k) - \val_2(k+1)$ is the number of nonterminal $1$s in the binary expansion of $k$; we leave it to the reader to convince herself by inspection that the expression $v(k+1) - 1$ computes the same thing. We hence have expressions for the $2$-valuations of all factors appearing in the product whose $2$-valuation we need to bound. We now make heavy use of a triangle inequality for $v(k)$, namely thatfor any nonnegative integers, we have $$v(k + j) \le v(k) + v(j).$$ Indeed, this triangle inequality immediately follows from Kummer's theorem for the $2$-valuation for $\binom{k + j}{k}$. The observation $v(2k) = v(k)$ shall also be important. Note for example that this inequality implies that \begin{eqnarray*} \val_2 C_{k + c} &=& v(k+c+1) - 1 \\ &\le& v(k+1) + v(c) - 1 \\ &=& \val_2 C_k + v(c). \end{eqnarray*} We now claim that replacing the pair $(a,c)$ by the pair $(a + c,0)$ can only reduce the $2$-valuation of the term we are estimating, i.e. that \begin{eqnarray*} \val_2 \Big( 2^c \binom{2\ell}{a,b,c,d,e,f,g} C_{a + (d + e)/2} C_{b + (d + e)/2} \Big) \ge \\ \, \val_2 \binom{2\ell}{a+c,b,d,e,f,g} C_{a + c + (d + e)/2} C_{b + (d + e)/2}.\end{eqnarray*} Indeed, the multinomial coefficient on the left side is simply a multiple of the multinomial coefficient on the right by $\binom{a + c}{a}$ and so certainly has (nonstrictly) larger $2$-valuation. On the other hand, $$\val_2 C_{a + c + (d + e)/2} \le \val_2 C_{a + (d + e)/2} + v(c) \le \val_2 C_{a + (d + e)/2} + c.$$ Hence, the claim is demonstrated. As our new septuple still satisfies our pair of conditions, we may therefore assume $c = 0$. Note also that in the last step of the above, the only way we can have equality is if $v(c) = c$, i.e. if $c = 0$ or $1$. 

It remains to show that for a sextuple $(a,b,d,e,f,g)$ satisfying $$a + b + d + e + f + g = 2\ell, d + 2f = e + 2g,$$ where we now denote $j = (d + e)/2 + 1$ for simplicity, that we have the inequality $$v(\ell) \le v(a) + v(b) + v(d) + v(e) + v(f) + v(g) - v(2\ell) + v(a + j) - 1 + v(b + j) - 1.$$ Rearranging, we wish to show $$2v(\ell) + 2 \le v(a) + v(b) + v(d) + v(e) + v(f) + v(g) + v(a + j) + v(b + j).$$ To show this inequality, we use our triangle inequality for the $v(k)$ function in the following way: \begin{eqnarray*} v(a) + v(b) + v(d) + v(e) + v(f) + v(g) + v(a+j) + v(b+j) &=& \\ \Big( v(a) + v(b+j) + v(d) + v(2f) \Big) + \Big( v(b) + v(a + j) + v(e) + v(2g) \Big) &\ge& \\ v(a+b+j+d+2f) + v(b+a+j+e+2g). && \end{eqnarray*} Using our constraint that $d + 2f = e + 2g$, we see that the two given expressions are in fact equal; moreover, their sum is \begin{eqnarray*} 2a + 2b + 2j + d + e + 2f + 2g &=& 2a + 2b + 2d + 2e + 2 + 2f + 2g \\ &=& 4\ell + 2 \end{eqnarray*} so that in fact each of the two given expressions is equal to $2\ell + 1$ and we have \begin{eqnarray*} v(a) + v(b) + v(d) + v(e) + v(f) + v(g) + v(a + j) + v(b + j) &\ge& 2v(2\ell+1) \\ &=& 2 \Big( v(\ell) + 1 \Big),\end{eqnarray*} which is precisely what we wanted to show.

We now show that we have an odd number of terms in the sum with $2$-valuation exactly equal to $v(\ell)$. Note that $$(a,d,f) \leftrightarrow (b,e,g)$$ is an involution we can perform on the septuples satisfying our two constraints which does not change the $2$-valuation of the term. As such, any septuples which aren't fixed under this involution may be paired up with their image; since we only care about checking we have an odd number of septuples with $2$-valuation exactly equal to $v(\ell)$, we may further suppose they are fixed points of this involution, i.e. that they additionally satisfy $$a = b, d = e, f = g.$$ Recall now that we have previously shown that unless $c$ is $0$ or $1$, the ensuing term will have $2$-valuation strictly larger than the $2$-valuation of some other term (where $c$ is replaced by zero) and so in particular cannot have the minimal possible $2$-valuation of $v(\ell)$. So we know that we must have $c = 0$ or $1$, but $a + b + c + d + e + f + g = 2\ell$ and $a = b, d = e, f = g$ implies $c$ is even, so we must in fact have $c = 0$. We may now rewrite our term as $$(-2)^c \binom{2\ell}{a,b,c,d,e,f,g} C_{a + (d+e)/2} C_{b + (d+e)/2} = \binom{2\ell}{\ell} \Big( \binom{\ell}{a,d,f} \Big)^2 C_{a+d}^2,$$ but recall that $$\val_2 \binom{2\ell}{\ell} = v(\ell)$$ so the other factors must be odd. By Kummer's theorem, this implies that $a, d, f$ have disjoint supports in their binary representations, i.e. that no two can have a $1$ in the same spot. As they sum to $\ell$, they form a partition of the $1$s in the binary representation of $\ell$. Now, in order for the Catalan number to be odd, we must have that $a + d$ is of the form $2^r - 1$. For any $r$ that works (i.e. such that the binary representation of $\ell$ ends with at least $r$ ones), $f$ is fixed as $\ell - 2^r + 1$ while $a$ and $d$ are allowed to partition the $1$s in the binary representation of $2^r - 1$ however they feel like it; in other words, for any subset of $\{0, 1, \cdots, r - 1\}$, we have a solution with $a$ given by the binary representation with $1$s exactly at members of that subset. As such, we have $2^r$ solutions, which is an even number except for the one case where $r = 0$, which corresponds to $f = \ell, a = d = 0$. In total, then, we have now shown that we have an odd number of terms with $2$-valuation equal to $v(\ell)$, which concludes the proof.

\bibliographystyle{alpha}
\bibliography{ref}
\end{document}